\documentclass[11pt]{article}

\usepackage{amssymb}
\usepackage{amsmath}
\usepackage[a4paper,margin=2.35cm]{geometry}
\usepackage{amsfonts}
\usepackage{amsthm}
\usepackage[english]{babel}
\usepackage{tikz}
\usepackage{graphicx}
\usepackage{subcaption}
\usepackage{bm}
\usepackage{bbm}
\usepackage{fancybox}
\usepackage{mathtools}
\usepackage{xcolor}

\usepackage{bookmark}
\DeclareRobustCommand{\vect}[1]{\bm{#1}}
\usepackage{ThesisNotationPaper} 
\usepackage{supertabular}
\usepackage{url}
\usepackage{bbm, xspace, lipsum}
\usepackage{color}
\usepackage[utf8]{inputenc}
\usepackage{emptypage}
\usepackage[final]{pdfpages}
\usepackage{dsfont} 

\usepackage{tikz}
\usetikzlibrary{arrows.meta,calc,arrows}
\usepackage{pgfplots}
\pgfplotsset{width=7cm,compat=1.3}

\usepackage{diagbox} 
\usepackage{subcaption} 
\usepackage{multirow}
\usepackage{booktabs}
\usepackage{varwidth}
\usepackage{tabularx}
\usepackage{authblk}

\begin{document}

\title{Large Fluctuations in Locational Marginal Prices}

\author[1]{T. Nesti}
\author[2]{J. Moriarty}
\author[3]{A. Zocca}
\author[1,4]{B. Zwart}

\affil[1]{\small CWI, Amsterdam 1098 XG, NL}
\affil[2]{\small Queen Mary University, London E1 4NS, UK}
\affil[3]{\small VU, Amsterdam 1081 HV, NL}
\affil[4]{\small TU/e, Eindhoven 5612 AZ, NL}





\maketitle 

\begin{abstract}
This paper investigates large fluctuations of Locational Marginal Prices (LMPs) in wholesale energy markets caused by 
  volatile renewable generation profiles. 
   Specifically, we study events of the form
   
\begin{equation}
\Prob \Big ( \bLMP \notin \prod_{i=1}^n [\alpha_i^-, \alpha_i^+] \Big),
\end{equation}
where $\bLMP$ is the vector of LMPs at the $n$ power grid nodes, and $\balpha^-,\balpha^+\in\Real^n$ are vectors of price thresholds specifying undesirable price occurrences.
By exploiting the structure of the supply-demand matching mechanism in power grids, we look at LMPs as deterministic piecewise affine, possibly discontinuous functions of the stochastic input process, modeling uncontrollable renewable generation. We utilize techniques from large deviations theory to identify the most likely ways for extreme price spikes to happen, and to rank the nodes of the power grid in terms of their likelihood of experiencing a price spike.
Our results are derived in the case of Gaussian fluctuations, and are validated numerically on the IEEE 14-bus test case.
\end{abstract}



\section{Introduction}

Modern-day power grids are undergoing a massive transformation, a prominent reason being the increase of intermittent renewable generation registered in the first two decades of the 21st century~\cite{Ren212019}.
The inherently uncertain nature of renewable energy sources like wind and solar photovoltaics is responsible for significant amounts of variability in power output, with important consequences for energy markets operations.
In particular, energy prices can exhibit significant volatility throughout different hours of the day, and are usually negatively correlated with the amount of renewable generation in the grid~\cite{Paraschiv2014}. 
In this paper, we focus on the Locational Marginal Pricing (LMP) mechanism~\cite{Ferc2003}, a market architecture adopted by many US energy markets. 
Under the LMP market architecture prices are location-dependent, and the presence of congested transmission lines 
causes them to vary wildly across different locations, contributing to their erratic behavior. On the other hand, under the zonal pricing mechanism~\cite{Neuhoff2011} (as in most European markets) a single price is calculated for each zone in the market. 

The topic of energy price forecasting has received a lot of attention in the forecasting community in the last 20 years, since the restructuring of energy markets from a government-controlled system to a deregulated market~\cite{Weron2014}.
Thanks to the particularly rich mathematical structure of the LMP mechanism, prediction models for LMPs are not limited to traditional statistical analysis and stochastic model-based techniques, but include structural methods exploiting the mathematical properties of the supply-demand matching process performed by grid operators, known as the Optimal Power Flow (OPF) problem~\cite{Hunueault1991}.

The relevant literature on structural prediction models can be categorized based on whether it takes an operator-centric~\cite{Bo2009,Ji2017} 
or a participant-centric~\cite{Zhou2011,Geng2016,Radovanovic2019holistic} 
point of view.

In the former case, it is assumed that the modeler has full knowledge of all the parameters defining the OPF formulation, such as generation cost functions, grid topology, and physical properties of the network. In~\cite{Bo2009}, the authors analyze the uncertainty in LMPs with respect to total load in the grid, relying on the structural property that changes in LMPs occur at the so-called critical load levels. 
In~\cite{Ji2017}, both load and generation
uncertainty is considered, and a multiparametric programming approach is proposed.

The market participant-centric approach, conversely, 
relies only on publicly available data, usually limited to grid-level, aggregated demand and generation, and nodal price data, without assuming knowledge of the network parameters. 
In~\cite{Zhou2011}, the authors utilizes the structure of the OPF formulation to infer the congestion status of transmission lines based only on zonal load levels, while
 in~\cite{Geng2016} a semi-decentralized data-driven approach, 
based on learning nodal prices as a function of nodal loads using support vector machines, is proposed. 
In~\cite{Radovanovic2019holistic}, a fully decentralized forecasting algorithm combining machine learning techniques with structural properties of the OPF is presented, and validation on the Southwest Power Pool market data results in accurate day-ahead predictions of real-time prices.

The methodologies described above have varying levels of performance in predicting expected intra-day variations, but they all have limitations when predicting extreme \textit{price spike} values. Even when assuming a a fully centralized perspective, forecasting price spikes is a notoriously
difficult problem~\cite{Lu2005}, and is mostly undertaken within the framework of zonal electricity markets~\cite{Lu2005,Hagfors2016,Paraschiv2016,Veraart2016}, 
while the corresponding problem for LMP-based markets has received less attention.
At the same time, the connection between locational marginal pricing and congestion status of the grid 
makes this problem particularly relevant for the discussion on financial transmission rights~\cite{Bushnell1999}, while a deeper understanding of the occurrence of high price events can inform network upgrades aimed at mitigating them~\cite{Wu2006}.

In this paper, we study the problem of predicting large price fluctuations in LMP-based energy markets from a centralized perspective, proposing a novel approach combining multiparametric programming techniques~\cite{Tondel2003} with large deviations theory~\cite{Dembo1998}. Large deviations techniques have been successfully used in fields such as queueing theory, telecommunication engineering, and finance~\cite{Bucklew1990}. In the recent years, they also have been applied in the context of power systems in order to study transmission line failures~\cite{Nesti2019temperature,Nesti2018emergent} and statistical properties of blackouts~\cite{Nesti2019blackout}.

In the present work we study the probability of nodal price spikes occurrences of the form
$\Prob\bigl( \bLMP \notin \prod_{i=1}^n [\alpha_i^-, \alpha_i^+]  \bigr)$,
where $\bLMP$ is the vector of Locational Marginal Prices at the $n$ grid nodes and $\balpha^-,\balpha^+\in\Real^n$ are vectors of price thresholds specifying undesirable price occurrences.
Assuming full knowledge of the power grid parameters, we first derive the deterministic function linking the stochastic input process, modeling renewable generation, to the $\bLMP$ vector. This, in turn, allows us to use large deviations theory to identify the most likely ways for extreme LMP spikes to happen as a result of unusual volatile renewable generation profiles.
The large deviations approach offers a powerful and flexible framework that holistically combines the network structure and operation paradigm (the OPF) with a stochastic model for renewable generation. This approach enables us to: i) approximate the probability of price spikes by means of solving a deterministic convex optimization problem, ii) rank the nodes of the power grids according to their likelihood of experiencing price spike events, iii) handle the \textit{multimodal} nature of the LMP's probability distribution, and iv) relax the LICQ regularity condition, an assumption that is usually required in the relevant literature~\cite{Zhou2011,Bo2009,Li2009,Bo2012}.

%
%
%
%
%


%


The rest of this paper is organized as follows.  A rigorous formulation of the problem under consideration is provided in Section~\ref{s:model}, while a connection to the field of multiparametric programming is established in Section~\ref{s:mpt}. In Section~\ref{s:ldp}, we derive our main large deviations result relating the event of a rare price spike to the solution of a deterministic optimization problem, which is further analyzed in Section~\ref{s:optimization}. We illustrate the potential of the proposed methodology in Section~\ref{s:numerics} with a case study on the IEEE 14-bus test case and draw our conclusions in Section~\ref{s:conclusion}.

\section{System model and problem formulation}\label{s:model}
The power grid is modeled as a connected graph $\Graph=\Graph(\Nodes,\Edges)$, where the set of nodes $\Nodes$ represents the $n$ buses in the system, and the set of edges $\Edges$ model the $m$ transmission lines. We assume that $\Nodes = \Nodes_g \sqcup \Nodes_{\theta}$, with $|\Nodes_g|=n_g, |\Nodes_{\theta}|=n_{\theta}$, $n_g + n_{\theta} = n$, and where $\sqcup$ denotes a disjoint union.
Each bus $i\in\Nodes_g$ houses a traditional controllable generator $\gen_i$, while each bus $i\in\Nodes_{\theta}$ houses a stochastic uncontrollable generating unit $\theta_i$. Finally, we assume that a subset of nodes $\Nodes_d\subseteq \Nodes$ houses loads, with $|\Nodes_d|=n_d$.
We denote the vectors of conventional generation, renewable generation, and demand, as the vectors 
$\bgen\in\Real_{+}^{\Nodes_g}, \bw\in\Real_{+}^{\Nodes_{\theta}}$, and $\bdem\in\Real_{+}^{\Nodes_d}$, respectively.~\footnote{The notation $x\in\Real^{A}$ indicates that the entries in the $|A|$-dimensional vector $x$ are indexed by the set $A$.} 
%
To simplify notation, we extend the vectors $\bgen,\bw,\bdem$ to $n$-dimensional vectors $\bgenext,\bwext,\bdemext\in\Real^n$ by setting $\tilde{g}_i=0$ whenever $i\notin \Nodes_g$, and similarly for $\bwext$ and $\bdemext$.
The vectors of \textit{net power injections} and \textit{power flows} are denoted by $\bpow := \bgenext + \bwext -\bdemext\in\Real^n$ and $\bflow\in\Real^m$, respectively. 

To optimally match power demand and supply while satisfying the power grid operating constraints, the Independent System Operator (ISO) solves the Optimal Power Flow (OPF~\cite{Hunueault1991}) problem and calculates the optimal energy dispatch vector $\bg^*\in\Real^{n_g}$, as well as the vector of \textit{nodal prices} $\bLMP\in\Real^{n}$, as we will describe in Section~\ref{s:model}\ref{ss:lmp}.%
In its full generality, the OPF problem is a nonlinear, nonconvex optimization problem, which is difficult to solve~\cite{Bienstockbook}. 

In this paper we focus on the widely used approximation of the latter known as DC-OPF, which is based on the \textit{DC approximation}~\cite{Purchala2005}. 
The DC approximation
relates any zero-sum vector $\bpow$ of net power injections and the corresponding power flows $\bflow$ via the linear relationship $\bflow=\bPTDF \bpow$, 
where the matrix $\bPTDF\in\Real^{m\times n}$, known as the power transfer distribution factor (PTDF) matrix, encodes information on the grid topology and parameters, cf.~Section~\ref{ss:PTDF_derivation}.
The DC-OPF can be formulated as the following quadratic optimization problem:
\begin{alignat}{5}
& \underset{\bgen\in\Real^{n_g}}{\min} & & \,\,\sum_{i=1}^{n_g} J_i(g_i)= \frac{1}{2}\bgen^{\tr} \bH\bgen + \bh^\tr \bg \label{eq:obj}\\
& \, \, \, \, \, \text{s.t.}          & & \ones^{\tr} (\bgenext+\bwext - \bdemext)=0  &  &\qquad :\lmpen \label{eq:balance} \\
&					   &  &\bfminuslimit \le \bPTDF(\bgenext +\bwext - \bdemext)\le \bfpluslimit  &  &\qquad :\bmum,\bmup \label{eq:lines}\\
&					   & & \bgminuslimit\le \bgenext \le \bgpluslimit   &  &\qquad :\btaum,\btaup  \label{eq:gen}
\end{alignat}
where the variables are defined as follows:\\

\begin{supertabular}{lp{\textwidth}}
$\bH\in\Real^{n_g\times n_g}$ & diagonal positive definite matrix appearing in the quadratic term of~\eqref{eq:obj};\\
 $\bh\in\Real^{n_g}$ & vector appearing in the linear term of the objective function~\eqref{eq:obj} ;\\
$\bPTDF \in\Real^{m\times n}$ & PTDF matrix (see definition later, in Eq.~\eqref{eq:PTDF_der});\\
$\bfminuslimit,\bfpluslimit \in\Real^m$ & vector of lower/upper transmission line limits;\\
\temphiddennote{more details here, line limits are important. Also, maybe remove this tabular notation.}
$\bgminuslimit,\bgpluslimit \in \Real^{n_g}$ & vector of lower/upper generation constraints;\\
$\lmpen\in\Real$ & dual variable of the energy balance constraint;\\
$\bmum, \bmup \in \Real^m_+$ & dual variables of the transmission line constraints~\eqref{eq:lines}, \\ 
$\btaum, \btaup \in\Real^{n_g}_+$ & dual variables of the generation constraints~\eqref{eq:gen};\\
$\ones\in\Real^n$ & $n$-dimensional vector of ones.
\\
\end{supertabular} 
Following standard practice~\cite{Sun2010},
we model the generation cost function $J_i(\cdot)$ at $i\in\Nodes_g$ as an increasing quadratic function of $g_i$, and denote by $\bJ(\bgen) := \sum_{i=1}^n J_i(g_i)$ the aggregated cost. 

\subsection{Locational marginal prices}\label{ss:lmp}
In this paper, we focus on energy markets adopting the concept of Locational Marginal Prices (LMPs) as electricity prices at the grid nodes. Under this market architecture, 
the LMP at a specific node is defined as the marginal cost of optimally supplying the next increment of load at that particular node while satisfying all power grid operational constraints, and can be calculated by solving the OPF problem in Eqs.~\eqref{eq:obj}-\eqref{eq:gen}.
\temphiddennote{there was a note about the OPF having a unique solution. Check if it is needed}
More precisely, let $\bg^*$ and $J^*=\bJ(\bg^*)$ denote, respectively, the optimal dispatch and the optimal value of the objective function of the OPF problem in Eqs.~\eqref{eq:obj}-\eqref{eq:gen}, and let $\mathcal{L}$ be the Lagrangian function. The LMP at bus $i$ is defined as the partial derivative of $J^*$ with respect to the demand $d_i$, and is equal to the partial derivative of the Lagrangian with respect to demand $d_i$ evaluated at the optimal solution: 
\begin{equation}\label{eq:LMP_def}
\LMP_i=\frac{\partial{J^*}}{\partial d_i}=\frac{\partial{\mathcal{L}}}{\partial d_i}\Bigr\rvert_{\bgen^*}.
\end{equation}
%
%
Following the derivation in~\cite{Radovanovic2019holistic},
the LMP vector can be represented as 
\begin{equation}\label{eq:LMP_dec}
\bLMP=\lmpen\ones+\bPTDF^\top\bmu\in\Real^n,
\end{equation}
where $\bmu=\bmum-\bmup$. 
Note that $\mu_{\ell}=0$ if and only if line $\ell$ is not congested, that is, if and only if $\fminuslimit_{\ell}<f_{\ell}<\fpluslimit_{\ell}$. In particular, $\mu^+_{\ell}>0$ if $f_{\ell}=\fpluslimit_{\ell}$, and $\mu^-_{\ell}<0$ if $f_{\ell}=\fminuslimit_{\ell}$.\footnote{$\mu^-,\mu^+$ cannot be both strictly positive, since lower and upper line flow constraints cannot be  simultaneously binding.}
As a consequence, if there are no congested lines, the LMPs at all nodes are equal, i.e.,  $\LMP_i=\lmpen$ for every $i=1\,\ldots,n,$ and the common value $\lmpen$ in \eqref{eq:LMP_dec} is known as the marginal \textit{energy component}. The energy component $\lmpen$ reflects the marginal cost of energy at the reference bus. 
If instead at least one line is congested, the LMPs are not all equal anymore 
and the term $\tilde{\bpi}:=\bPTDF^\top\bmu$ in Eq.~\eqref{eq:LMP_dec} is called the marginal \textit{congestion component}.
When ISOs calculate the LMPs, they also include a \textit{loss
component}, which is related to the heat dissipated on
transmission lines and is not accounted for by the DC-OPF model. The loss component is typically negligible compared to the other price components~\cite{SPPreport2016}\temphiddennote{mention $2\%$ thing?}, and its inclusion goes beyond the scope of this paper.
%

\subsection{Problem statement}\label{ss:prob_statement}
In this paper, we adopt a \textit{functional} perspective, i.e., we view the uncontrollable generator as a variable parameter, or input, of the OPF. 
In particular, we are interested in a setting where the objective function, PTDF matrix, nodal demand $\bdem$, line limits and generation constraints are assumed to be known and fixed. Conversely, the uncontrollable generation $\bw\in\Real^{\Nodes_{\theta}}$ corresponds to a \textit{variable parameter} of the problem, upon which the solution of the OPF problem in Eqs.~\eqref{eq:obj}-\eqref{eq:gen} (to which we will henceforth refer as $\OPF(\btheta)$), and thus the LMP vector, depend.

In other words, the LMP vector is a deterministic function of $\btheta$
\begin{equation}\label{eq:mapping}
\Real^{n_{\theta}} \supseteq \bTheta \ni \btheta \rightarrow \bLMP(\btheta)\in\Real^n,
\end{equation}
where $\bTheta \subseteq \Real^{n_{\theta}}$ is the \textit{feasible parameter space} of the OPF, i.e., the set of parameters $\btheta$ such that $\OPF(\btheta)$ is feasible. 
In particular, we model $\btheta$ as a non-degenerate multivariate Gaussian vector $\btheta_{\eps} \sim \mathcal{N}_{n_{\theta}}(\bmuth,\eps\bSigmath)$,
where the parameter $\eps>0$ quantifies the magnitude of the noise.
The mean $\bmuth\in \mathring{\bTheta}$ (where $\mathring{A}$ denotes the interior of the set $A$) of the random vector $\btheta$ is interpreted as the expected, or nominal, realization of renewable generation for the considered time interval. Furthermore, we assume that $\bSigmath$ is a known positive definite matrix, and consider the regime where $\eps\to 0$.
In view of the mapping~\eqref{eq:mapping}, $\bLMP$ is a $n$-dimensional random vector whose distribution depends on that of $\btheta$, and on the deterministic mapping $\btheta \to \bLMP(\btheta)$.
We assume that the $\bLMP$ vector corresponding to the expected renewable generation $\bmuth$ is such that
\begin{equation}\label{eq:condition_mu}
\bLMP(\bmuth) \in \mathring{\Pi},
 \end{equation}
where $\Pi:=\prod_{i=1}^n [\alpha_i^-, \alpha_i^+]$, and $\balpha^{-},\balpha^{+}\in\Real^n$ are vectors of price thresholds. We are interested in the event $Y=Y(\balpha^{-},\balpha^{+})$ of \textit{anomalous price fluctuations} (or \textit{price spikes}) defined as
\begin{align}
	Y(\balpha^{-},\balpha^{+}) &= \bigl\{ \btheta\in \bTheta \,:\, \bLMP (\btheta) \notin \prod_{i=1}^n [\alpha_i^-, \alpha_i^+]  \bigr\} \label{eq:PriceSpike}\\
	&=\bigcup_{i=1}^n \{  \btheta\in \bTheta \,:\, \LMP_i(\btheta) < \alpha_i^{-} 
	\text{ or }  \LMP_i(\btheta)  > \alpha_i^{+}\},\label{eq:PriceSpike_2}
\end{align}
which, in view of Eq.~\eqref{eq:condition_mu} and the regime $\eps \to 0$, is a \textit{rare event}. Without loss of generality, we only consider thresholds $\balpha^{-},\balpha^{+}$ such that the event $Y(\balpha^{-},\balpha^{+})$ has a non-empty interior in $\Real^{n_{\theta}}$. Otherwise, the fact that $\bmuth$ is non degenerate would imply 
$\Prob \bigl(Y(\balpha^{-},\balpha^{+})\bigr) = 0$.  

We observe that the above formulation of a price spike event is quite general, and can cover different application scenarios, as we now outline.
For example, if $\alpha = \alpha^{+}_i=-\alpha^{-}_i > 0$ for all $i$, then the price spike event becomes 
\[Y(\alpha) =\{\btheta\in\Theta\, : \, \norm{\bLMP}=\max_{i=1,\ldots,n} |\LMP_i| > \alpha\},\]
and models the occurrence of a price spike larger than a prescribed value $\alpha$. On the other hand, if we define 
$\balpha^{-}= \bLMP(\bmuth) -\bbeta$ and $\balpha^{+}= \bLMP(\bmuth) + \bbeta$, for $\bbeta\in\Real_{+}^{n}$, the spike event
\[Y(\bbeta) = \bigcup_{i=1}^n\{\btheta\in\bTheta\, : \, |\LMP_i -\LMP_i(\bmuth)|> \beta_i\},\]
models the event of any $\LMP_i$ deviating from its nominal value $\LMP_i(\bmuth)$ more than $\beta_i>0$. 
Moreover, by setting  $\balpha^{-}= \bLMP(\bmuth) -\bbeta^{-}$ and $\balpha^{+}= \bLMP(\bmuth) + \bbeta^{+}$, $\bbeta^{-},\bbeta^{+}\in\Real^{m,+}$ and $\bbeta^{-}\neq \bbeta^{+}$, we can weigh differently negative and positive deviations from the nominal values.

We remark that \textit{negative} price spikes are also of interest~\cite{Gerster2016,Gonzales2017} and can be covered in our framework, by choosing the threshold vectors $\balpha^{-},\balpha^{+}$ accordingly.
Finally, note that we can study price spikes at a more granular level by 
restricting the union in Eq.~\eqref{eq:PriceSpike_2} to a particular subset of nodes  $\widetilde{\Nodes}\subseteq \Nodes$.
%
%
%

\subsection{Derivation of the PTDF matrix $\bPTDF$}\label{ss:PTDF_derivation}
Choosing an arbitrary but fixed orientation of the transmission lines, the network topology is described by the \textit{edge-vertex incidence matrix} $\binc\in\Real^{m\times n}$ defined as
 $\inc_{\ell, i}=1$ if $\ell=(i,j)$, $\inc_{\ell, i}=-1$ if $\ell=(j,i)$, and $\inc_{\ell, i} = 0$ otherwise.
We associate to every line $\ell \in \Edges$ a \textit{weight} $\weight_{\ell}$, which we take to be equal to the inverse of the reactance $x_{\ell}>0$ of that line, i.e., $\weight_{\ell}=x_{\ell}^{-1}$~\cite{Bienstockbook}. Let $\bdiagdc\in\Real^{m\times m}$ be the diagonal matrix containing the line weights $\bdiagdc=\diag(\weight_{1}^,\dots, \weight_{m})$.
The network topology and weights are simultaneously encoded in the \textit{weighted Laplacian matrix} of the graph $\Graph$, defined as $\blap = \binc^\tr \bdiagdc \binc$.
Finally, by setting node $1$ as the reference node,
the PTDF matrix is given by
\begin{equation}\label{eq:PTDF_der}
\bPTDF :=[\zero\,\, \bdiagdc\bincsub\blapsub^{-1}],
\end{equation}
where $\bincsub\in\Real^{m\times(n-1)}$ is the matrix obtained by deleting the first columns of $\binc$, $\blapsub^{(n-1)\times (n-1)}$ by deleting the first row and column of $\blap$.

\section{Multiparametric programming}\label{s:mpt}

As discussed in Section~\ref{s:model}\ref{ss:prob_statement}, LMPs can be thought as deterministic functions of the parameter $\btheta$. Therefore, in order to study the distribution of the random vector $\bLMP$, we need to investigate the structure of the mapping $\btheta \to \bLMP(\btheta)$.
We do this using the language of Multiparametric Programming Theory (MPT)~\cite{Tondel2003}, which is concerned with the study of optimization problems which depend on a vector of parameters, and aims at analyzing the impact of such parameters on the outcome of the problem, both in terms of primal and dual solutions. In our setting, we stress that the parameter $\btheta $
models the uncontrollable renewable generation. Hence, the problem $\OPF(\btheta)$ can be formulated as a standard Multiparametric Quadratic Program (MPQ) as follows:
\begin{align}
\underset{\bgen \in \Real^{n_g}}{\min} &\quad  \frac{1}{2}\bgen^\tr \bH^\tr \bgen + \bgen^\tr \bh \qquad \label{eq:OPF_MQP_1} \\
 \text{s.t.}        &\quad \bA\bgen\le \bb+\bE\btheta,   \label{eq:OPF_MQP_2} 
\end{align}
where $\bA\in\Real^{(2+2m+2n_g)\times n_g},\bE\in\Real^{(2+2m+2n_g)\times n_{\theta}},\bb\in\Real^{(2+2m+2n_g)}$ are defined as
\begin{equation}\label{eq:MQP}
\bA=\begin{bmatrix}
\phantom{-}\mathbf{1}_{n_g}^\tr\\
-\mathbf{1}_{n_g}^\tr\\
 \phantom{-}\bPTDF_{\Nodes_g}\\
 -\bPTDF_{\Nodes_g}\\
 \phantom{-}\bI_{n_g} \\
-\bI_{n_g}\\
\end{bmatrix},\quad
\bb=\begin{bmatrix}
  \phantom{-}\ones^\tr \bdem\\
  -\ones^\tr\bdem\\
  \phantom{-}\bPTDF_{\Nodes_d} \bdem +\bfpluslimit\\
  -\bPTDF_{\Nodes_d} \bdem -\bfminuslimit\\
  \phantom{-}\bgpluslimit\\
  -\bgminuslimit\\ 
\end{bmatrix},\quad
\bE=\begin{bmatrix}
 -\ones_{n_{\theta}}^\tr   \\
   \phantom{-}\ones_{n_{\theta}}^\tr \\
-\bPTDF_{\Nodes_{\theta}}   \\
    \phantom{-}\bPTDF_{\Nodes_{\theta}}  \\
   \phantom{-}\zeros_{n_{\theta}}  \\
   \phantom{-}\zeros_{n_{\theta}} \\
   \end{bmatrix}.
\end{equation}
For $k\in\Nat$, denote by $\ones_k, \zeros_k\in\Real^k$ and $\bI_{k}\in\Real^{k\times k}$ the vector of ones, zeros, and the identity matrix of dimension $k$, respectively.  Moreover, $\bPTDF_{\Nodes_g}\in\Real^{m\times n_g}$ and $\bPTDF_{\Nodes_{\theta}}\in\Real^{m\times n_{\theta}}$ denote the submatrices of $\bPTDF$ obtained by selecting only the columns corresponding to nodes in $\Nodes_g$ and $\Nodes_{\theta}$, respectively.

A key result in MPT~\cite{Tondel2003} is that
the feasible parameter space $\bTheta\subseteq \Real^{n_{\theta}}$ of the problem Eqs.~\eqref{eq:OPF_MQP_1}-\eqref{eq:OPF_MQP_2} can be partitioned into a finite number of convex polytopes, each corresponding to a different \textit{optimal partition}, i.e., a grid-wide state vector that indicates the saturated status of generators and congestion status of transmission lines. 
\begin{definition}[Optimal Partition]\label{def:partition}
Given a parameter vector $\btheta\in\bTheta$, 
let $\bg^*=\bg^*(\btheta)$ denote the optimal generation vector obtained by solving the problem defined by Eqs.~\eqref{eq:OPF_MQP_1}-\eqref{eq:OPF_MQP_2}. Let $\mathcal{J}$ denote the index set of constraints in Eq.~\eqref{eq:OPF_MQP_2}, with $|\mathcal{J}|=2+2m+2n$. 
The \textit{optimal partition} of $\mathcal{J}$ associated with $\btheta$ is the partition 
$\mathcal{J} = \B(\btheta)\sqcup \B^{\complement}(\btheta)$, 
with $\mathcal{B}(\btheta)= \{i\in\mathcal{J}\,\,|\, \bA_i\bgen^*=\bb+\bE_i\theta\}$ and $
\mathcal{B}^\complement(\btheta)=\{i\in\mathcal{J}\,|\, \bA_i\bg^*<\bb+\bE_i\btheta\}$.
\end{definition} 
The sets $\mathcal{B}$ and $\mathcal{B}^\complement$, respectively, correspond to binding and non-binding constraints of the OPF and, hence, identify congested lines and nonmarginal generators. With a minor abuse of notation, we identify the optimal partition $(\B,\B^\complement)$ with the corresponding set of binding constraints $\B$. Given an optimal partition $\B$, let $\bA_{\B},\bE_{\B}$ denote the submatrices of $\bA$ and $\bE$ containing the rows $\bA_i,\bE_i$ indexed by $i\in\B$, respectively. 

\begin{remark}\label{rm:redundant}
The energy balance equality constraint~\eqref{eq:balance} in the original OPF formulation is rewritten as two inequalities indexed by $i=1,2$ in Eq.~\eqref{eq:OPF_MQP_2}, which are always binding and read
$\bA_i\bg^*=\bb_i+\bE_i\btheta$, $i=1,2$.
Looking at Eq.~\eqref{eq:MQP}, we see that the two equations $\bA_i\bg^*=\bb_i+\bE_i\btheta, \, i=1,2,$ are identical, and thus one of them is redundant. In the rest of this paper, we eliminate one of the redundant constraints from the set $\B$, namely the one corresponding to $i=2$. Therefore, we write $\B = \{1\} \sqcup \Bcong \sqcup \Bsat$, where $\Bcong\subseteq\{3,\ldots,2+2m\}$ describes the congestion status of transmission lines, and $\Bsat\subseteq \{2+2m+1,2+2m+2n\}$ describe the saturated status of generators.
 \end{remark}
\begin{definition}[LICQ]\label{def:licq}
Given an optimal partition $\B$, we say that the \textit{linear independent constraint qualification} (LICQ) holds if
the matrix $\bA_{\B}\in\Real^{|\B|\times n}$ has full row rank.
\end{definition}
Since there is always at least one binding constraint, namely $i=1$ (corresponding to the power balance constraint), we can write $|\B| = 1 + |\B\pr|$, where $\B\pr = 
\Bcong \sqcup \Bsat$ contains the indexes of binding constraints corresponding to line and generator limits. 
Since line and generation limits cannot be binding both on the positive and negative sides, we have that $|\B\pr| = |\Bsat| + |\Bcong|\le n_g + m$. Moreover, it is observed in~\cite{Zhou2011} that the row rank of $\bA_{\B}$ is equal to 
$ \min(1+|\Bsat|+|\Bcong|, n_g)$,
implying that the LICQ condition is equivalent to  
\begin{equation}\label{eq:LICQ_eq}
1+|\Bsat|+|\Bcong|\le n_g.
\end{equation}
The following theorem is a standard result in MPT, see~\cite[Theorem 1]{Tondel2003}. It states that there exist $M$ affine maps defined in the interiors of the critical regions
\[\mathring{\bTheta_k} \ni \btheta \to \bLMP_{\rvert \mathring{\bTheta_k}} (\btheta) = \bCtildek \btheta + \bctildek,\, k=1\ldots,M\]
where $\bCtildek,\bctildek$ are suitably defined matrices and vectors. Moreover, if LICQ holds for every $\btheta\in\bTheta$, then the maps agree on the intersections between the regions $\bTheta_k$'s, resulting in an overall continuous map 
$
\bTheta \ni \btheta \to \bLMP(\btheta) \in \Real^n.
$
\begin{theorem}\label{th:mpt1}
Assume that $\bH$ is positive definite, $\bTheta$ a fully dimensional compact set in $\Real^{n_{\theta}}$, and that the LICQ regularity condition is satisfied for every $\btheta\in\bTheta$. Then, $\bTheta$ can be covered by the union of a finite number $M$ 
of fully-dimensional compact convex polytopes $\bTheta_1,\ldots,\bTheta_M$, referred to as \textit{critical regions}, such that: 
(i) their interiors are pairwise disjoint
$\mathring{\bTheta}_k \cap \mathring{\bTheta}_h = \emptyset$ for every $k\neq h$,
and each interior $\mathring{\bTheta}_k $ corresponds to the largest set of parameters yielding the same optimal partition;
(ii) within the interior of each critical region $\mathring{\bTheta}_k$, the optimal generation $\bgen^*$ and the associated $\bLMP$ vector are affine functions of $\btheta$,
(iii) the map $\bTheta\ni\btheta \to \bLMP(\btheta)$ defined over the entire parameter space is piecewise affine and continuous.%
\end{theorem}

\subsection{Relaxing the LICQ assumption}
One of the assumptions of Theorem~\ref{th:mpt1}, which is standard in the literature~\cite{Zhou2011,Bo2009,Li2009,Bo2012}, is that the LICQ condition holds for every $\btheta\in\bTheta$.
In particular, this means that LICQ holds in the interior of two neighboring regions, which we denote as $\mathring{\bTheta}_i$ and $\mathring{\bTheta}_j$. Let $\Hyp$ be the hyperplance separating $\mathring{\bTheta}_i$ and $\mathring{\bTheta}_j$. The fact that LICQ holds at $\mathring{\bTheta}_i$ implies that, if $\{i_1,\ldots,i_q\}$ are the binding constraints at optimality in the OPF for $\btheta\in\mathring{\bTheta}_i$,
  then in view of Eq.~\eqref{eq:LICQ_eq} we have $q\le n_g$, where we recall that $n_g$ is the number of decision variables in the OPF (i.e., the number of controllable generators).

Requiring LICQ to hold everywhere means that, in particular, it must hold in the 
common facet between regions.  As we move from $\mathring{\bTheta}_k$ on to the common facet $\F = \bTheta_i \cap \Hyp$ between regions $\bTheta_i$ and $\bTheta_j$, which has dimension $n_{\theta}-1$, there could be an additional constraint becoming active (coming from the neighboring region $\bTheta_j$), and therefore the LICQ condition implies $q+1\le n_g$. In general, critical regions can intersect in faces of dimensions $1,\ldots,n_{\theta}-1$, and enforcing LICQ to hold on all these faces could imply the overly-conservative assumption $q+n_{\theta}-1\le n_g$. 

In what follows, we relax the assumptions of Theorem~\ref{th:mpt1} by allowing LICQ to be violated on the union of these lower-dimensional faces
\begin{equation}\label{eq:Theta0}
\bTheta_{\circ} :=\bTheta \setminus  \bigcup_{k=1}^M \mathring{\bTheta}_k.
\end{equation}
Since this union has zero $n_{\theta}$-dimensional Lebesgue measure, the event $\btheta\in\bTheta_{\circ}$ rarely happens in practice, and thus is usually ignored in the literature, but it does cause a technical issue that we now address.
If LICQ is violated on $\btheta \in \bTheta_{\circ}$, the Lagrange multipliers of the OPF, and thus the LMP, need not be unique. Therefore, the map $\btheta \to \bLMP(\btheta)$ is not properly defined on $\bTheta_{\circ}$.
 In order to extend the map from $\bigcup_{k=1}^M \mathring{\bTheta}_k$ to the full feasible parameter space $\bTheta$, we incorporate a tie-breaking rule to consistently choose between the possible LMPs. 
Following~\cite{Tang2013nash}, we break ties by using the lexicographic order. 
This choice defines the LMP function over the whole feasible parameter space $\bTheta$, but may introduce jump discontinuities on the zero-measure set $\bTheta_{\circ}$. In the next section, we address this technicality and formally derive our main large deviations result. 
%
\section{Large deviations results}\label{s:ldp}
\begin{proposition}\label{prop:ldp}
Let $\btheta_{\eps} \sim \mathcal{N}_{n_{\theta}}(\bmuth,\eps\bSigmath)$ be a family of nondegenerate $n_{\theta}$-dimensional Gaussian r.v.'s indexed by $\eps>0$. Assume that the LICQ condition is satisfied for all $\btheta\in\bTheta\setminus  \bTheta_{\circ}$. Consider the event 
\[Y=Y(\balpha^-,\balpha^+) = \bigcup_{i=1}^n \{  \btheta\in \bTheta \,:\, \LMP_i(\btheta) \notin [\alpha_i^{-},\alpha_i^{+}]\},
\]
defined in Eq.~\eqref{eq:PriceSpike}, assume that the interior of $Y$ is not empty~\footnote{~If $\mathring{Y}=\emptyset$, then trivially $\Prob (\btheta_{\eps} \in Y )= 0$.} and that 
\begin{equation}
\bLMP(\bmuth) \in \mathring{\Pi},\quad \Pi:=\prod_{i=1}^n [\alpha_i^-, \alpha_i^+].
 \end{equation}
Then, the family of random vectors $\{\btheta_{\eps}\}_{\eps>0}$ satisfies 
\begin{equation}\label{eq:ldp_th}
\lim_{\eps \to 0} \eps \log \Prob (\btheta_{\eps} \in Y) = -\inf_{\btheta \in Y} I(\btheta),
\end{equation}
where $\I(\btheta) = \frac{1}{2}(\btheta -\bmuth)^\top \bSigmath^{-1} (\btheta -\bmuth).$. 
\end{proposition}

\begin{proof}
For notational compactness, in the rest of the proof we will write $Y$ without making explicit its dependence on $(\balpha^-,\balpha^+)$. Defining  $Z := \bigcup_{i=1}^n \{\btheta \in\Real^{n_{\theta}} \,:\ \LMP_i(\btheta) \notin [\alpha_i^{-},\alpha_i^{+}]\}$, the event $Y$ can be decomposed as the disjoint union
$Y = Y_*\cup Y_{\circ}$, where
\begin{align}
Y_*=
\bigcup_{k=1}^M \mathring{\bTheta}_k \cap Z,\quad
 Y_{\circ} =  \bTheta_{\circ} \cap Z,
\end{align}
and $Y_{\circ} \subseteq \bTheta_{\circ} = \bTheta \setminus  \bigcup_{k=1}^M \mathring{\bTheta}_k$ is a zero-measure set.
As $\btheta_{\eps}$ is non nondegenerate, it has a density $f$ with respect to the $n_{\theta}$-dimensional Lebesgue measure in $\Real^n_{\theta}$. Since the $n_{\theta}$-dimensional Lebesgue measure of $ Y_{\circ}$ is zero, we have
\[\Prob (\btheta_{\eps} \in Y_{\circ})=\int_{\bx\in Y_{\circ}} f(\bx) d\bx =0\]
and $\Prob (\btheta_{\eps} \in Y) = \Prob (\btheta_{\eps} \in Y_*)$.
As a consequence, we can restrict our analysis to the event $Y_*$.
Thanks to Cramer's theorem in $\Real^{n_{\theta}}$~\cite{Dembo1998}, we have 
\begin{align}\label{eq:LDP-general}
-\inf _{\btheta \in \mathring{Y}_* }I(\btheta)
\leq &\liminf_{\eps\to 0}\eps\log {\big (}{\Prob}(\btheta_{\eps} \in Y_*){\big )}\\
\leq &\limsup _{\eps\to 0}\eps\log {\big (}{\Prob}(\btheta_{\eps} \in Y_*){\big )}\leq -\inf_{\btheta \in \overline{Y_*}}I(\btheta),
\end{align}
where $I(\btheta)$ is the Legendre transform of the log-moment generating function of $\btheta_{\eps}$. It is well-known (see, for example,~\cite{Touchette2009}) that when $\btheta_{\eps}$ is Gaussian then $\I(\btheta) = (\btheta -\bmuth)^\tr \bSigmath^{-1} (\btheta -\bmuth)$.
In order to prove~\eqref{eq:ldp_th}, it remains to be shown that  
\begin{equation}\label{eq:two_infs}
\inf _{\btheta \in \mathring{Y}_*}I(\btheta) = \inf_{\btheta \in \overline{Y_*}}I(\btheta).
\end{equation}
Thanks to the continuity of the maps $\bLMP\rvert_{\mathring{\bTheta}_k}$, the set $Y_*$ is open, since
\begin{align}\label{eq:union1}
Y_* = &\bigcup_{k=1}^M \mathring{\bTheta_k} \cap Z 
=\bigcup_{k=1}^M \bigl(\mathring{\bTheta_k} \cap  \bigcup_{i=1}^n\{\ \LMP_i\rvert_{\mathring{\bTheta}_k}(\btheta) \notin [\alpha_i^{-},\alpha_i^{+}] \}\bigr)\\
=&\bigcup_{k=1}^M \Bigr(\mathring{\bTheta_k} \cap  \bigcup_{i=1}^n
\{\bCtildek_i \btheta + \bctildek_i \notin [\alpha_i^{-},\alpha_i^{+}]\}\Bigl)\\
%
%
=&\bigcup_{k=1}^M \bigcup_{i=1}^n \Bigr(\mathring{\bTheta_k} \cap  
(\{\bCtildek_i \btheta + \bctildek_i < \alpha_i^-\} \cup \{\bCtildek \btheta + \bctildek > \alpha_i^+\})\Bigl).
%
\end{align}
Therefore,  $\mathring{Y_*} = Y_*$, $\overline{\mathring{Y_*}}  = \overline{Y_*}\supseteq Y_*$ and Eq.~\eqref{eq:two_infs} follows from the continuity of $I(\btheta)$.
\end{proof}
Proposition~\ref{prop:ldp} allows us approximate the probability of a price spike, for small $\eps$, as 
\begin{equation}\label{eq:ldp_approx}
\Prob (\btheta_{\eps} \in Y) \approx \exp\,\Bigl(\frac{ -\inf_{\btheta \in Y} I(\btheta)}{\eps}\Bigr), 
\end{equation}
as it is done in~\cite{Nesti2019temperature,Nesti2018emergent} in the context of studying the event of transmission line failures.
Moreover, the minimizer of the optimization problem~\eqref{eq:ldp_th} corresponds to the most likely realization of uncontrollable generation that leads to the rare event. 
Furthermore, the structure of the problem~\eqref{eq:ldp_th}  allows us to efficiently rank nodes in terms of their likelihood to experience a price spike, as we illustrate in Section~\ref{s:numerics}.

\section{Solving the optimization problem}\label{s:optimization}
In view of Proposition~\ref{prop:ldp}, in order to study
$\lim_{\eps \to 0} \eps \log \Prob (\btheta_{\eps} \in Y(\balpha^-,\balpha^+))$ 
we need to solve the deterministic optimization problem
$\inf_{\btheta \in Y_{*}} I(\btheta).$ The latter, in view of Theorem~\ref{th:mpt1} and the definition of $Y_*$, the latter is equivalent to
\begin{align*}
\inf_{\btheta \in Y_{*}} I(\btheta) = 
&\min_{k=1,\ldots,M} \inf_{\btheta \in \mathring{\Theta}_k \cap Z} I(\btheta)
=\min_{i=1,\ldots,n }\min_{k=1,\ldots,M}  \inf_{\btheta \in \mathring{\Theta}_k, \bCtildek_i \btheta + \bctildek_i\notin [\alpha_i^-,\alpha_i^+]} I(\btheta).
\end{align*}
This amounts to solving at most $nM$ quadratic optimization problems of the form $\inf_{\btheta \in T_{i,k}} I(\btheta)$
for $i=1,\ldots,n$, $k=1,\ldots,M$, where 
 \[ T_{i,k} = T^-_{i,k} \sqcup T^+_{i,k},\quad 
 T^-_{i,k} = \mathring{\Theta}_k \cap \ \{ \bCtildek_i \btheta + \bctildek_i< \alpha_i^- \},\quad
  T^+_{i,k} =\mathring{\Theta}_k \cap \{ \bCtildek_i \btheta + \bctildek_i> \alpha_i^+ \}.\] 

In the rest of this section, we show how we can significantly reduce the number of optimization problems that need to be solved by exploiting the geometric structure of the problem. 
First, since
\begin{align*}
\inf_{\btheta \in Y_{*}} I(\btheta) 
=&\min_{i=1,\ldots,n }\inf_{\btheta \in \bigcup_{k=1}^M T_{i,k}} I(\btheta),
\end{align*}
we fix $i=1\ldots, n$ and consider the sub-problems
\begin{equation}\label{eq:problem_fixed_i} 
\inf_{\btheta \in \bigcup_{k=1}^M T_{i,k}} I(\btheta) = 
\min\,\Bigl\{ 
\inf_{\btheta \in \bigcup_{k=1}^M T^-_{i,k}} I(\btheta),
\inf_{\btheta \in \bigcup_{k=1}^M T^+_{i,k}} I(\btheta)
\Bigr\}.
\end{equation}
The reason why we want to solve the problems in Eq.~\eqref{eq:problem_fixed_i} individually for every $i$ is because we are not only interested in studying the overall event $Y$, but also in the more granular events of node-specific price spikes. For example, this would allow us to rank the nodes in terms of their likelihood of experiencing a price spike (see Section~\ref{s:numerics}).
\temphiddennote{There may be something smarter that can be done in order to only solve the original problem, by reversing the order or the discrete minimums.}
Define 
\begin{align*}
&L^-_{(i,k)} :=  \bTheta_k \cap (T^-_{i,k})^{\complement}= 
\bTheta_k \cap \{\bCtildek_i \btheta + \bctildek_i \ge \alpha_i^{-}\},\quad
L_i^- := \bigcup_{k=1}^M  L^-_{i,k} = \bTheta \cap \{\LMP_i \ge \alpha_i^-\},\\
& L^+_{(i,k)} :=  \bTheta_k \cap (T^+_{i,k})^{\complement}= 
\bTheta_k \cap \{\bCtildek_i \btheta + \bctildek_i \le\alpha_i^{+}\},\quad
L_i^+ := \bigcup_{k=1}^M  L^+_{i,k} = \bTheta \cap \{\LMP_i \le \alpha_i^+\},
\end{align*}
and consider the partition of the sets $L_i^+$ and $L_i^-$ into disjoint closed connected components, i.e.,
\begin{equation}
L_i^-= \bigsqcup_{\ell \in \text{conn. comp. of } L_i^- } W^{(i,-)}_{\ell},\,
L_i^+= \bigsqcup_{\ell \in \text{conn. comp. of } L_i^+ }  W^{(i,+)}_{\ell},
\end{equation}
and let  $W^{(i,-)}_{\ell^{-*}},W^{(i,+)}_{\ell^{+*}}$ be the components containing $\bmuth$.
Since $\partial(A \cup B) = \partial A \cup \partial B$ if $\overline{A}\cap B = A \cap \overline{B} = \emptyset$, the boundary $\partial L_i^+ = \bigsqcup_{\ell \in \F^+_{i}} \partial W^{(i,+)}_{\ell}$ is the union of the set 
of parameters $\btheta\in\bTheta$ such that $\bLMP(\btheta)=\alpha_i^+$ 
with, possibly, a subset of the boundary of $\bTheta$ (and similarly for $\partial L_i^-$).

As stated by Proposition~\ref{prop:opt_1}, we show that, in order to solve the two problems in the right hand side of Eq.~\eqref{eq:problem_fixed_i}
we need to look only at the boundaries
$\partial W^{(i,-)}_{\ell^*},\partial W^{(i,+)}_{\ell^*}$.

\begin{proposition}\label{prop:opt_1}
Under the same assumptions of Theorem~\ref{th:mpt1}, we have
\begin{align}
&\inf_{\btheta \in \bigcup_{k=1}^M T^{+}_{i,k}} I(\btheta) = 
\inf_{\btheta \in \partial \bigcup_{k=1}^M T^{+}_{i,k}} I(\btheta), \quad
 \inf_{\btheta \in \bigcup_{k=1}^M T^{-}_{i,k}} I(\btheta) = 
\inf_{\btheta \in \partial \bigcup_{k=1}^M T^{-}_{i,k}} I(\btheta).
\label{eq:opt_1}
\end{align}
Moreover, \begin{align}
&\inf_{\btheta \in \bigcup_{k=1}^M T^{+}_{i,k}} I(\btheta) = 
\inf_{\btheta \in \partial W^{(i,+)}_{\ell^{*+}}} I(\btheta),
\quad
\inf_{\btheta \in \bigcup_{k=1}^M T^{-}_{i,k}} I(\btheta) = 
\inf_{\btheta \in \partial W^{(i,-)}_{\ell^{*-}}} I(\btheta).\label{eq:lake_minus}
\end{align}
\end{proposition}

\begin{proof}
First note that the rate function $I(\btheta)$ is a (strictly) convex function, since $\bSigmath$ is positive definite.
Since $\bigcup_{k=1}^M T^{+}_{i,k}$ is open and $I(\btheta)$ is a continuous function, it holds that
\[\inf_{\btheta \in \bigcup_{k=1}^M T^{+}_{i,k}} I(\btheta) = 
\inf_{\btheta \in  \overline{\bigcup_{k=1}^M T^{+}_{i,k}}} I(\btheta) .\]
Moreover, since $I(\btheta)$ is continuous and
$\overline{\bigcup_{k=1}^M T^{+}_{i,k}}$ compact, the infimum is attained. 
The fact that $\overline{\bigcup_{k=1}^M T^{+}_{i,k}}\supseteq  \partial \bigcup_{k=1}^M T^{+}_{i,k}$ immediately implies that 
\[\inf_{\btheta \in \bigcup_{k=1}^M T^{+}_{i,k}} I(\btheta) \le
\inf_{\btheta \in \partial \bigcup_{k=1}^M T^{+}_{i,k}} I(\btheta).\]
On the other hand, assume by contradiction that
\[\inf_{\btheta \in \bigcup_{k=1}^M T^{+}_{i,k}} I(\btheta) <
\inf_{\btheta \in \partial \bigcup_{k=1}^M T^{+}_{i,k}} I(\btheta).\]
In particular, there exists a point $\btheta_0$ in the interior of $\bigcup_{k=1}^M T^{+}_{i,k}$ such that  $I(\btheta_0)<I(\btheta)$ for all $\btheta \in \bigcup_{k=1}^M T^{+}_{i,k}$. 
Define, for $t\in[0,1]$, the line segment joining $\bmuth$ and $\btheta_0$, i.e. $\btheta_{t} = (1-t)\bmuth + t \btheta_0$.
 Since $\btheta_0$ lies in the interior of $\bigcup_{k=1}^M T^{+}_{i,k}$, and $\bmuth\notin \bigcup_{k=1}^M T^{+}_{i,k}$, there exist a $0< t_*<1$ such that $\btheta_t \in \bigcup_{k=1}^M T^{+}_{i,k}$ for all $t\in [t_*,1]$.
 Due to the convexity of $I(\btheta)$, and the fact that $I(\bmuth)=0$,
 we have 
 \[I(\btheta_{t_*}) < (1-{t_*}) I(\bmuth) + {t_*} I(\btheta_0) = {t_*} I(\btheta_0) < I(\btheta_0),\] thus reaching a contradiction.
 Hence, \begin{equation}
\inf_{\btheta \in \bigcup_{k=1}^M T^{+}_{i,k}} I(\btheta) = 
\inf_{\btheta \in \partial \bigcup_{k=1}^M T^{+}_{i,k}} I(\btheta),
\end{equation} 
and the minimum is achieved on $ \partial \bigcup_{k=1}^M T^{+}_{i,k}$, proving Eq.~\eqref{eq:opt_1}.


In view of Eq.~\eqref{eq:opt_1}, in order to prove Eq.~\eqref{eq:lake_minus} it is enough to show that 
\[\inf_{\btheta \in \partial \bigcup_{k=1}^M T^{+}_{i,k}} I(\btheta)=
\inf_{\btheta \in \partial W^{(i,+)}_{\ell^{*+}}} I(\btheta).\]

Given that the sets  $T^{+}_{i,k}=\mathring{\bTheta}_k \cup \{\LMP_i(\btheta)>\alpha_i^+\}$, for $k=1,\ldots,M$, are disjoint, the boundary of the union is equal to the union of the boundaries, i.e., 
$\partial \bigcup_{k=1}^M T^{+}_{i,k}  =  \bigcup_{k=1}^M \partial T^{+}_{i,k}.$ 
Each term $\partial T^{+}_{i,k}$ is the boundary of the polytope 
$\overline{T^{+}}_{i,k} = \bTheta_k \cap \{\LMP_i \ge \alpha_i^+\}$, and thus 
consists of the union of a subset of $\bigcup_{k=1}^M \partial \bTheta_k$ 
(a subset of the union of the facets of the polytope $\bTheta_k$) with the segment $\bTheta_k \cap \{\LMP_i = \alpha_i^+\}.$
As a result, $\partial \bigcup_{k=1}^M T^{+}_{i,k} I(\btheta)$
intersects $\partial W^{(i,+)}_{\ell^{*+}}$ in $\bTheta \cap \{\LMP_i=\alpha_i^+\}$.

We now show that (i) the minimum is attained at a point $\btheta_0$ such that $\LMP_i(\btheta_0) = \alpha_i^+$, so that
 $\btheta_0 \in \bigsqcup_{\ell \in \text{conn. comp. of } L_i^+ } \partial W^{(i,+)}_{\ell}$, 
 and (ii) $\btheta_0\in  \partial W^{(i,+)}_{\ell^{*+}}$.
Assume by contradiction that $\LMP_i(\btheta_0)>\alpha_i^+$, and consider the line segment joining $\bmuth$ and $\btheta_0$, $\btheta_{t} = (1-t)\bmuth + t \btheta_0, \, t\in[0,1]$.
The function
\[ [0,1] \ni t \to g(t) := \LMP_i (\btheta_t) = \LMP_i ((1-t)\bmuth + t \btheta_0) \in \Real,\]
 is continuous and such that $g(0) = \bLMP(\bmuth)< \alpha$ and
  $g(1)=\bLMP(\btheta_0)>\alpha_i^*$. Thanks to the intermediate value theorem, there exists a $0<t_*<1$ 
  such that $g(t_*)=\LMP_i(\btheta_{t_*})=\alpha_i^*$, and 
 \[I(\btheta_{t_*}) < (1-t_*) I(\bmuth) + t_* I(\btheta_0) =  t_* I(\btheta_0) < I(\btheta_0),\]
which is a contradiction, since $\btheta_0$ is the minimum.
 Therefore, 
  $\btheta_0\in \bigsqcup_{\ell \in \text{conn. comp. of } L_i^+ } \partial W^{(i,+)}_{\ell}$.  The same argument, based on the convexity of the rate function and the fact that $I(\bmuth)=0$, shows that $\btheta\in \partial W^{(i,+)}_{\ell^{*,+}}$. 
Lastly, Eq.~\eqref{eq:lake_minus} can be derived in the same way.
\end{proof}

Proposition~\ref{prop:opt_1} shows that in order to solve the problem in Eq.~\eqref{eq:problem_fixed_i} we only need to look at the boundaries
$\partial W^{(i,+)}_{\ell^*}, \partial W^{(i,-)}_{\ell^*}$. Determining such boundaries is a non-trivial problem, for which dedicated algorithms exist. However, such algorithms are beyond the scope of this paper and we refer the interested reader to the contour tracing literature and, in particular, to~\cite{Dobkin1990contour}.

\section{Numerics}\label{s:numerics}
In this section, we illustrate the potential of our large deviations approach using the standard IEEE 14-bus test case from MATPOWER~\cite{Zimmerman2011}. This network consists of 14 nodes housing loads, 6 controllable generators, and 20 lines. As line limits are not included in the test case, we set them as $\bfpluslimit = \lambda \bfpluslimit^{(\text{planning)}}$, where $ \bfpluslimit^{(\text{planning})} :=  \gamma_{\text{line}} |\bflow|$, $\bflow$ is the solution of a DC-OPF using the data in the test file, and $\gamma_{\text{line}}\ge 1$. 
We interpret $\bfpluslimit^{(\text{planning})}$ as the maximum allowable power flow before the line trips, while $\lambda$ is a safety tuning parameter satisfying $1/\gamma_{\text{line}}\le\lambda\le 1$.
In the rest of this section, we set $\gamma_{\text{line}}=2$ and $\lambda=0.6$. 

We add two uncontrollable renewable generators at nodes $4$ and $5$, so that $n_d=14,n_g=6$ and $n_{\theta}=2$. 
All the calculations related to multiparametric programming are performed using the MPT3 toolbox~\cite{MPT3}. 
We model the renewable generation as a 2-dimensional Gaussian  random vector $\btheta\sim \mathcal{N}_2(\bmuth,\bSigmath)$, where $\bmuth$ is interpreted as the nominal, or forecast, renewable generation. The covariance matrix $\bSigmath$ is calculated as in~\cite{Hoffmann2019consistency} to model positive correlations between neighboring (thus geographically ``close'') nodes. More specifically, we consider normalized symmetric graph Laplacian $L_{\text{sym}} = \Delta ^{-1/2} L_{\text{sym}} \Delta^{-1/2},$ where
$\Delta\in\Real^{n\times n}$ is the diagonal matrix with entries equal to $\Delta_{i,i}=\sum_{j\neq i} \weight_{i,j}$. We then compute the matrix
\begin{equation}\label{eq:cov_IEEE14}
\bC = \tau^{2\kappa} (L_{\text{sym}}  + \tau^2 I)^{-\kappa}\in\Real^{n\times n},
\end{equation}
for $\kappa=2$ and $\tau^2=1$ as in~\cite{Hoffmann2019consistency}, 
\temphiddennote{write intuition on this matrix} 
\temphiddennote{not sure if it's true that $\kappa$ controls the amount of correlation}
and consider the $n_{\theta}\times n_{\theta}$ submatrix $\btildeSigmath$ obtained from of $\bC$ by choosing rows and columns indexed by $\Nodes_{n_{\theta}}=\{4,5\}$,
and we define $\bSigmath$ as
\begin{equation}
\bSigmath := \diag(\{\delta_i\}_{i=1}^{n_{\theta}}) \,\btildeSigmath\, \diag(\{\delta_i\}_{i=1}^{n_{\theta}})\,\in\Real^{n_{\theta}\times n_{\theta}},
\end{equation}
where the parameters $\delta_i$'s control the magnitudes of the standard deviations $\sigma_{i}:=\sqrt{\bSigmath(i,i)}$, $i=1,2$. In particular, the $\delta_i$'s are chosen 
in such a way that the standard deviations $\sigma_{i}$'s match realistic values for wind power forecasting error expressed as a fraction of the corresponding installed capacity, over different time windows $T$, namely
$\sigma_{i} = q(T) \times \mu_i^{(\text{installed})},\, i=1,2,$
where $q=[0.01,0.018,0.04]$, corresponding to time windows of $5,15$ and $60$ minutes, respectively (see~\cite{Nesti2019temperature}, Section V.B). 
 Finally, the installed capacity of the renewable generators are chosen based on the boundary of the $2$-dimensional feasible space $\bTheta$, namely
 $\mu_1^{(\text{installed})} = \max\{x\,:\, (x,y) \in \bTheta\}, 
 \mu_2^{(\text{installed})} = \max\{y\,:\, (x,y) \in \bTheta\}.$
 Although $\btheta\sim \mathcal{N}_2(\bmuth,\bSigmath)$ is in principle unbounded, we choose the relevant parameters in such a way that, in practice, $\btheta$ never exceeds the
boundary of the feasible space $\bTheta$.
\temphiddennote{comment here about $\eps$ not being present, like in the PRL paper?}
 Since $\bSigmath$ is obtained from realistic values for wind power forecasting error, the question is whether the matrix $\bSigmath$ used in the numerics is close enough to the small-noise regime to make the large deviations results meaningful. As we show, the answer to this question is affirmative, validating the use of the large deviations methodology.


\subsection{Multimodality and sensitivity with respect to $\vect{\mu}_{\vect{\theta}}$}
%
  %
Given $\bmuth \in \bTheta$, we set the price thresholds defining the spike event as $\balpha^{-}_i= \LMP_i(\bmuth) - \errrel |\LMP_i(\bmuth)| ,\,\balpha^{+}_i= \LMP_i(\bmuth) + \errrel |\LMP_i(\bmuth)|,$
where $\errrel>0$. In other words, we are interested in studying the event of a \textit{relative price deviation} of magnitude greater than $\errrel>0$:
{\footnotesize
\begin{equation}
Y=\bigcup_{i=1}^n Y_i, \qquad Y_i = \{\btheta\in\bTheta\,:\,
|\LMP_i(\btheta) -\LMP_i(\bmuth)| > \errrel  |\LMP_i(\bmuth)|\}.
\end{equation}
}

%

\temphiddennote{ridirlo dopo: In particular, we show how the results are extremely sensitive to the location of the forecast renewable generation $\bmuth$.}
Next, we
 consider two scenarios, corresponding to low and high expected wind generation, i.e.
 $\bmuth^{(\text{low})} = 0.1 \times\bmu^{(\text{installed})}$, $\bmuth^{(\text{high})} = 0.5 \times\bmu^{(\text{installed})}$,and variance corresponding to the 15-minute window, i.e. $q^{(\text{medium})} = 0.018$.
Fig.~\ref{fig:panel} shows the location of $\bmuth$, together with $10^6$ samples from $\btheta$, the corresponding empirical density of the random variable $\LMP_{10}$ obtained through Monte Carlo simulation, and a visualization of the piecewise affine mapping $\btheta\to \bLMP_{10}(\btheta)$.

We observe that the results are extremely sensitive to the standard deviation and location of the forecast renewable generation $\bmuth$ relative to the geometry of the critical regions, as this 
affects whether the samples of $\btheta$ will cross the boundary between adjacent regions or not. 
In turn, the crossing of a boundary can result in the distribution of the LMPs being \textit{multimodal} (see Fig.~\ref{fig:panel}, left panels), due to the piecewise affine nature of the map $\btheta\to \bLMP$. This observation shows how the problem of studying LMPs fluctuations is intrinsically harder than that of emergent line failures, as in~\cite{Nesti2017line,Nesti2018emergent}.
The phenomenon is more pronounced in the presence of steep gradient changes at the boundary between regions (or in the case of discontinuities), as can be observed in the right panels of Fig.~\ref{fig:panel}, which show the piecewise affine map $\btheta\to \bLMP_{10}(\btheta)$ for the two different choices of $\bmuth$.
In particular, the expected $\bLMP$ can differ greatly from $\bLMP(\bmuth)$. 
   

  \begin{figure}[!h]
       \centering
\begin{subfigure}[t]{\textwidth}
\includegraphics[width=0.32\textwidth]{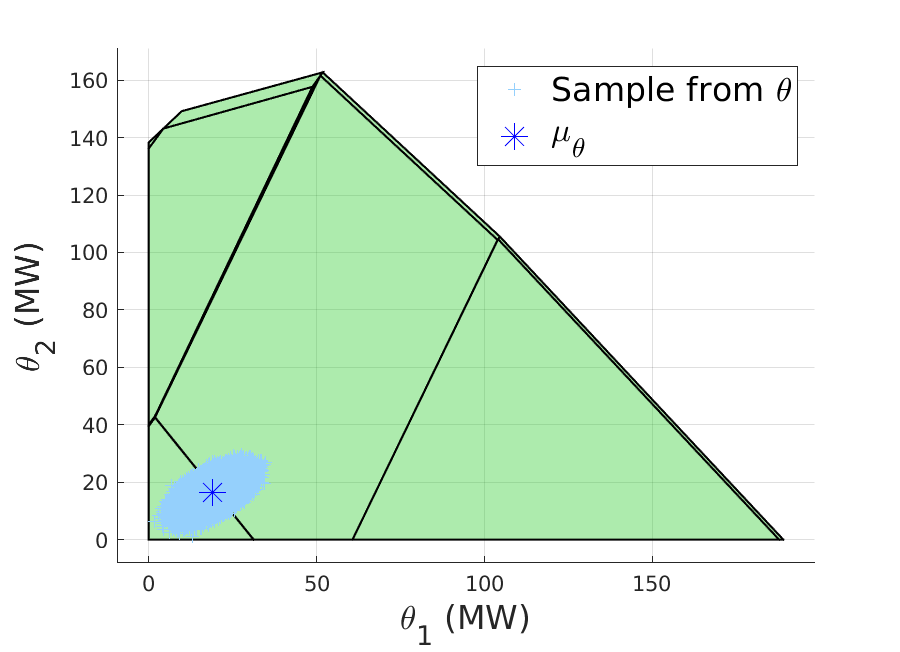} 
        %
        \includegraphics[width=0.32\textwidth]{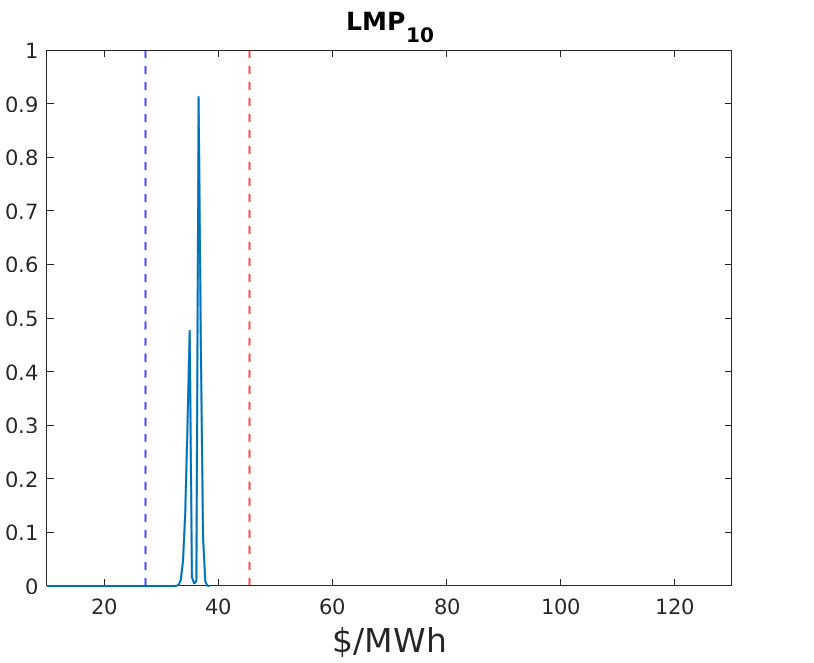}
                  %
\includegraphics[width=0.32\textwidth]{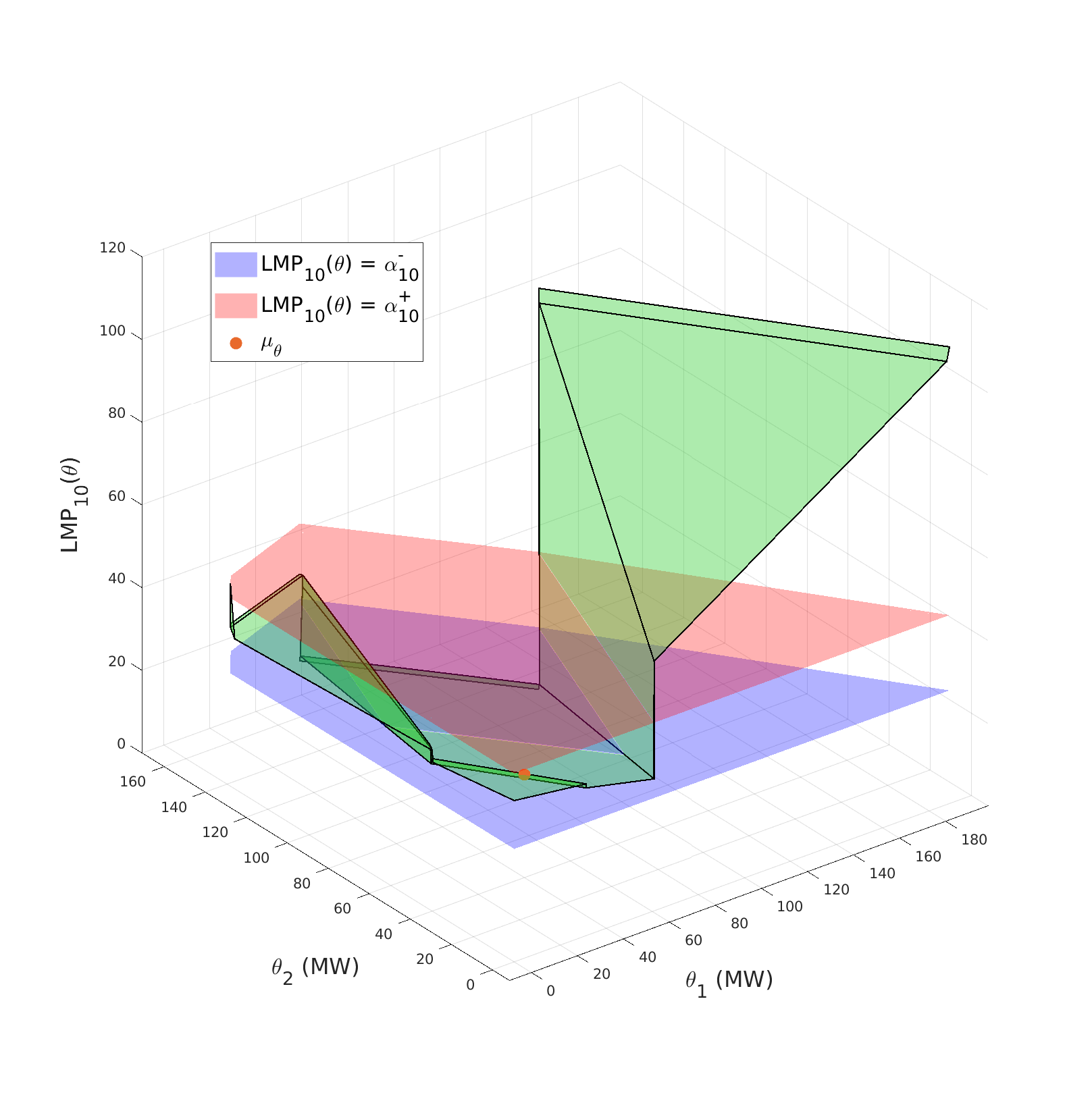}
\caption{\footnotesize $\bmuth^{(\text{low})} = 0.1 \times\bmu^{(\text{installed})},\errrel=0.25, q=0.018$.} 
 \label{panel_a}         
        \end{subfigure}
                 %
                 \vspace{-0.25cm}
 \begin{subfigure}[t]{\textwidth}
\includegraphics[width=0.32\textwidth]{
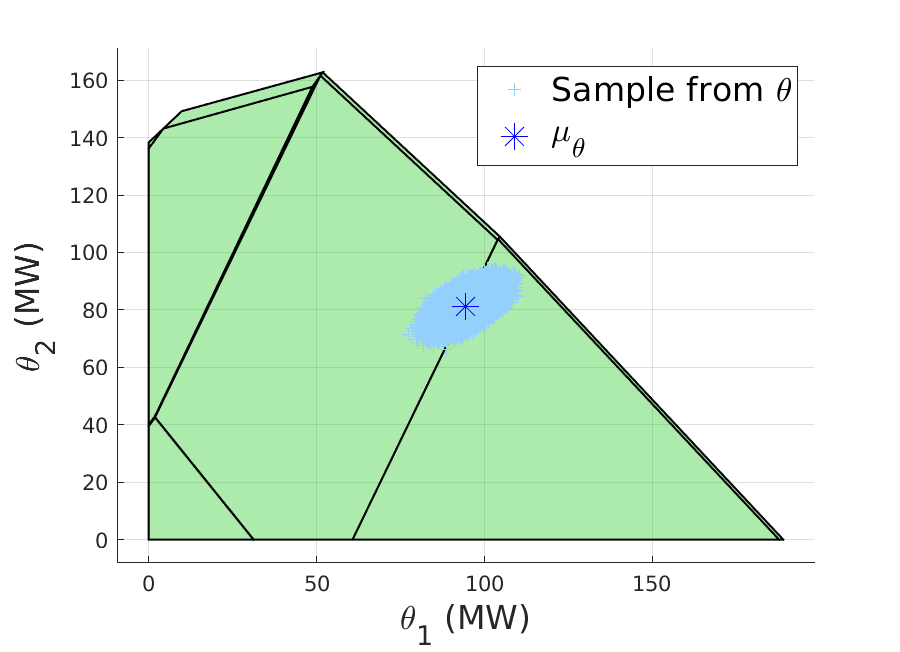}            
%
        \includegraphics[width=0.32\textwidth]{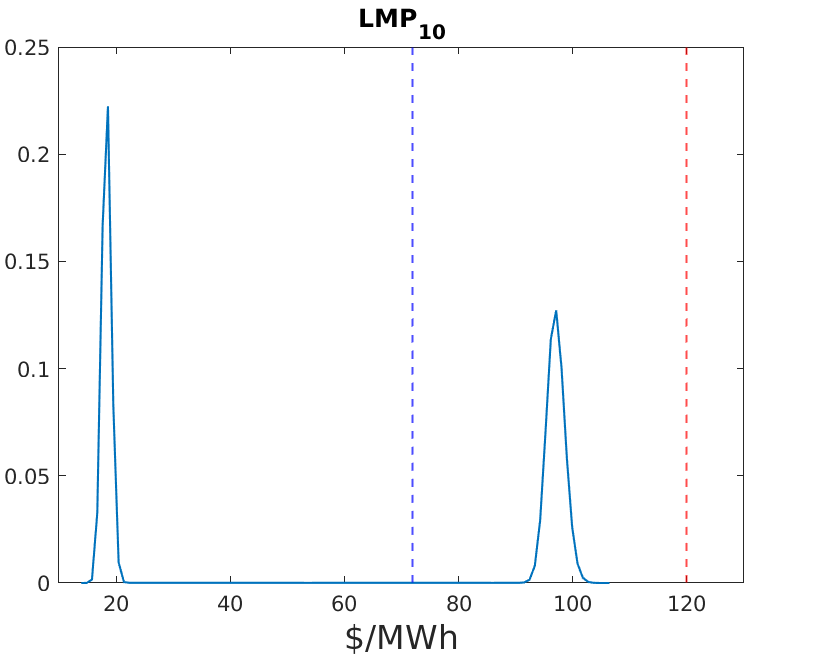}
\includegraphics[width=0.32\textwidth]{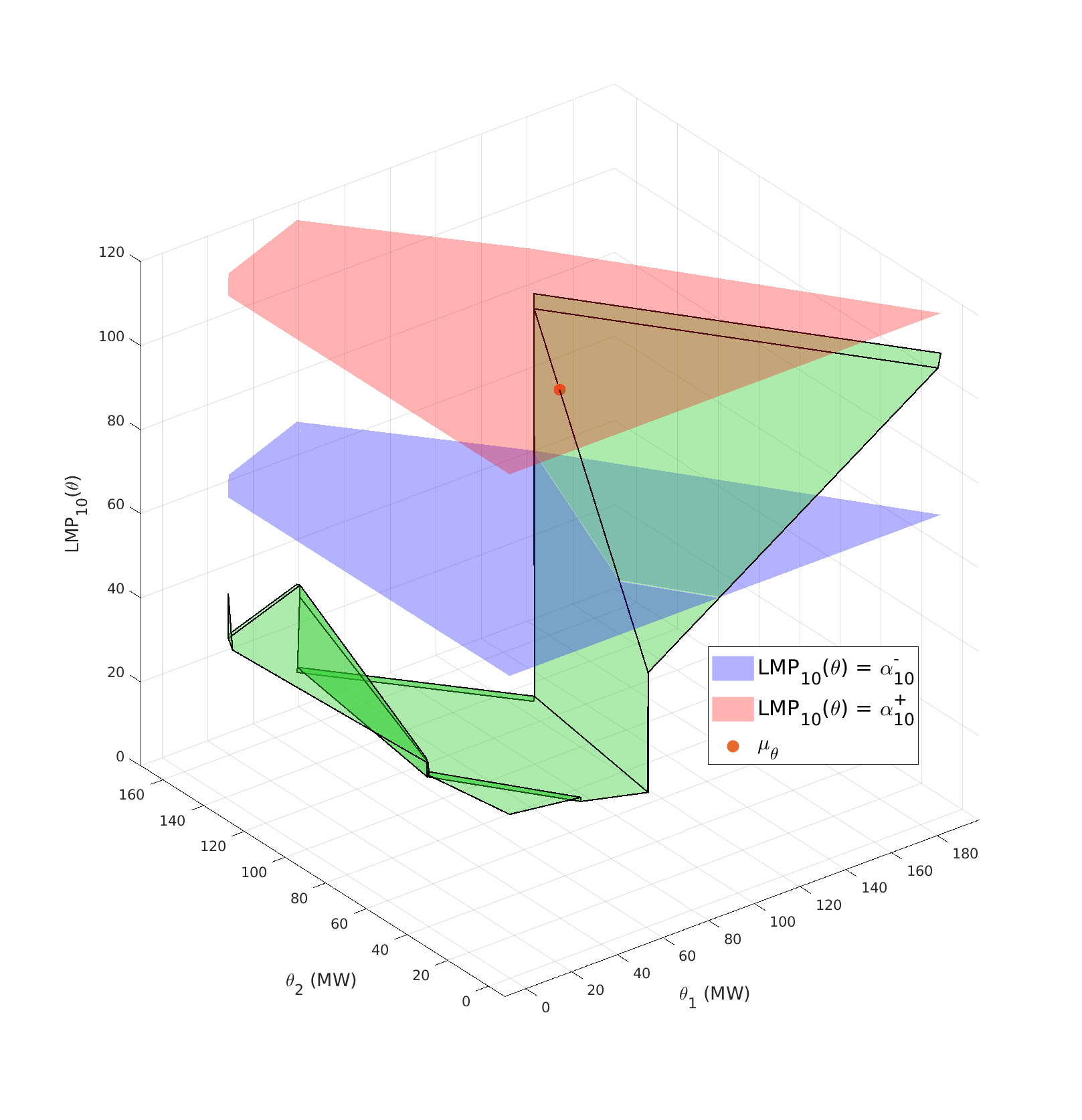}
                \caption{\footnotesize $\bmuth^{(\text{high})} = 0.5 \times\bmu^{(\text{installed})},\errrel=0.25, q=0.018$.}
                \label{panel_b}         
         \end{subfigure}
            %
 \caption{\footnotesize  forecast generation $\bmuth$ and empirical distribution of renewable generation $\btheta$ (left); Empirical density of the random variable $\LMP_{10}$, with the thresholds $\alpha^-_i$ and $\alpha^+_i$ represented as red and blue vertical bars, respectively (middle); Piecewise affine map $\btheta\to \bLMP_{10}(\btheta)$ with price thresholds $\alpha^{\pm}_{10}= \LMP_{10}(\bmuth) \pm \errrel |\LMP_{10}(\bmuth)|$, for two different choices of $\bmuth$ (right).
}
 \label{fig:panel}
\end{figure}

  \begin{figure}[!h]
       \centering
       \begin{subfigure}[t]{0.23\columnwidth}
\includegraphics[width=\textwidth]{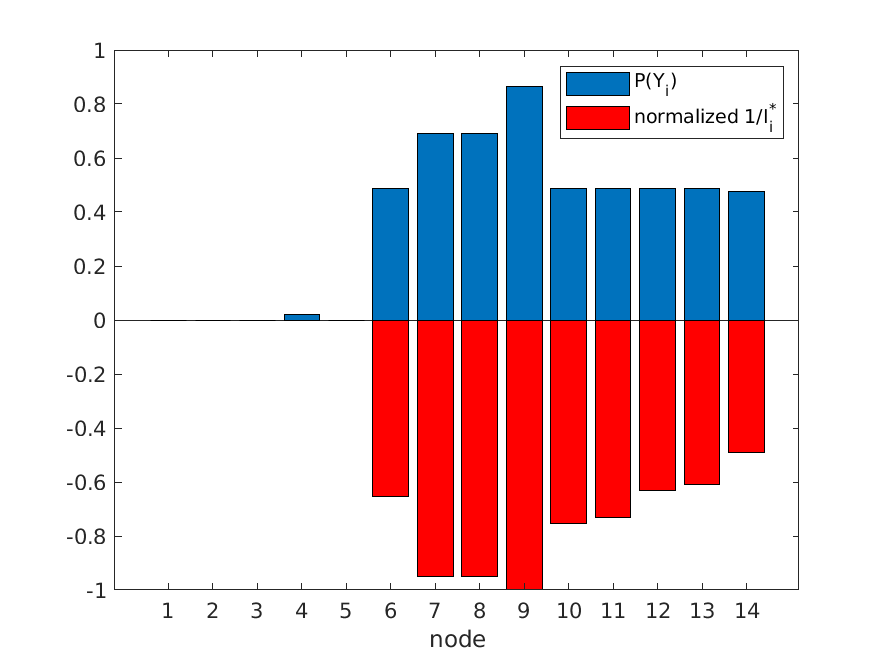}
    \caption{\footnotesize $\errrel=0.25$.}
            \end{subfigure}
            ~
                 \begin{subfigure}[t]{0.23\columnwidth}
\includegraphics[width=\textwidth]{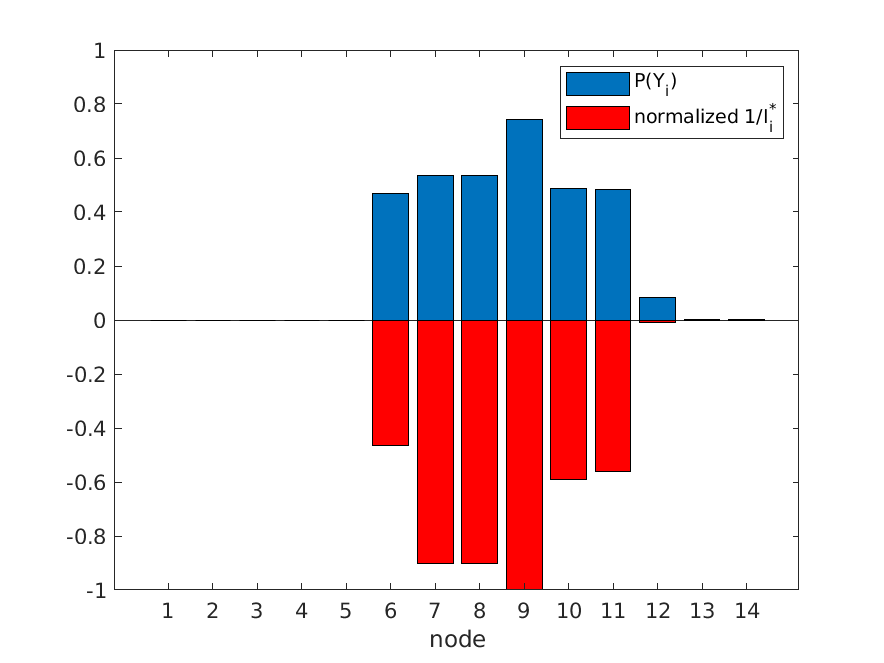}
    \caption{\footnotesize $\errrel=0.5$.}
            \end{subfigure}
                     ~
                 \begin{subfigure}[t]{0.23\columnwidth}
\includegraphics[width=\textwidth]{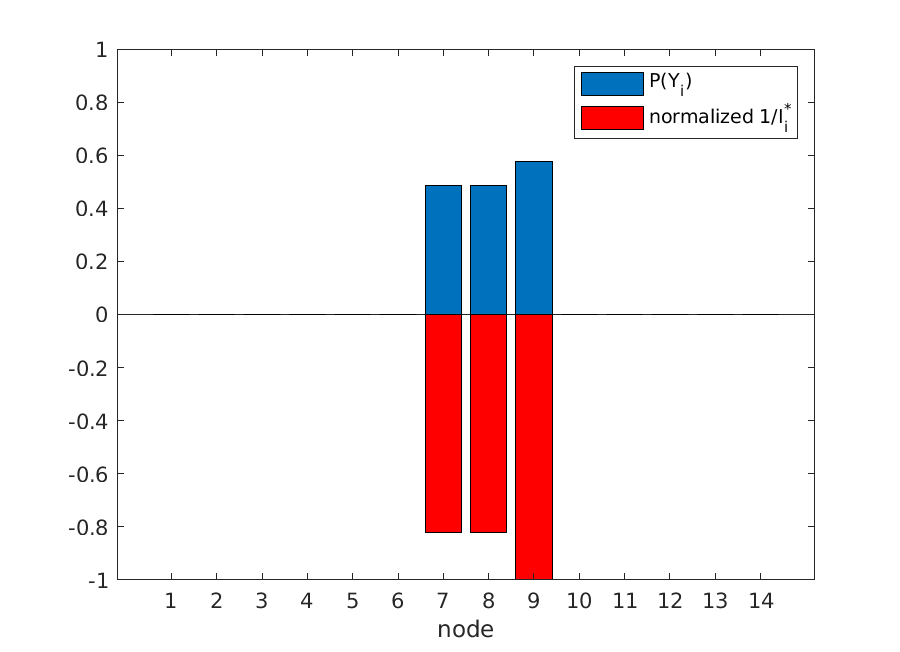}
    \caption{\footnotesize $\errrel=1$.}
            \end{subfigure}
                ~
\begin{subfigure}[t]{0.23\columnwidth}
\includegraphics[width=\textwidth]{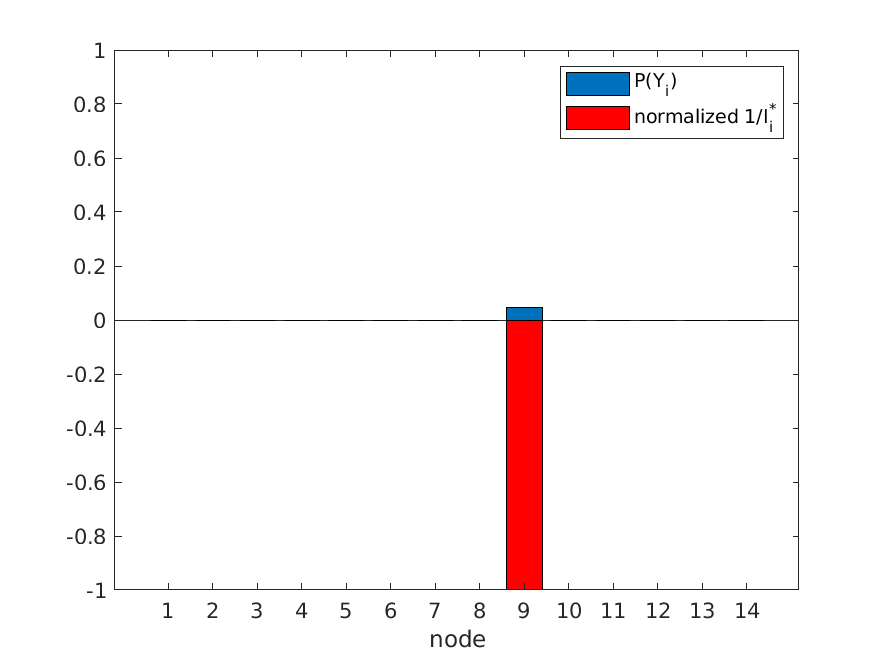}
    \caption{\footnotesize $\errrel=10$.}
            \end{subfigure}
 \caption{Comparison between empirical probabilities $\widehat{\mathbb{P}}(Y_i)$ based on Monte Carlo simulation and normalized decay rates $-\min_{i} I^*_i/I^*_i$ for various level of $\errrel.$}
 \label{fig:_bars}
\end{figure}

\subsection{Ranking of nodes based on their likelihood of having a price spike}
As illustrated by Eq.~\eqref{eq:problem_fixed_i}, large deviations theory predicts the most likely node to be
$\argmin_{i=1,\ldots,n} I^*_i,$
where  $I^*_i := \inf_{\btheta \in \bigcup_{k=1}^M T_{i,k}} I(\btheta)$. Indirectly, this approach produces also a \textit{ranking} of nodes according to their likelihood of having a price spike. The use of large deviations theory to rank power grid components according to their likelihood of experiencing anomalous deviations from a nominal state has been validated in~\cite{Nesti2018emergent} in the context of transmission line failures.
In order to validate the accuracy of the LDP methodology also for ranking nodes according to the likelihood of their price spikes, we compare the LD-based ranking with the one obtained via crude Monte Carlo simulation, as described in Table~\ref{tab:_ranking_1}. We observe that the LD-based approach is able to recover the exact ranking of nodes, for various levels of relative error $\errrel$.
Table~\ref{tab:_ranking_1} reports the values of the probability $\widehat{\mathbb{P}}(Y_i)$ of a price spike in node $i$, calculated using Monte Carlo simulation, together with the corresponding decay rates  $I^*_i = \inf_{\btheta \in \bigcup_{k=1}^M T_{i,k}} I(\btheta)$, showing that the LD-based approach correctly identifies the ranking.
This property is validated more extensively in Fig.~\ref{fig:_bars}, which depicts the values of 
$\widehat{\mathbb{P}}(Y_i)$ against $-\min_{k} I^*_k/I^*_i$ across a wider range of price thresholds $\errrel$, corresponding to decreasing probability of the price spike event.

 \begin{table}[h!]
\centering
\begin{tabular}{c|c|c|c}
\hline
\hline
$i$ & $ \widehat{\mathbb{P}}(Y_i)$ & $ I^*_i $ & \text{rank}\\
\hline
9   & 8.6371e-01  & 	   8.1160e-04		&  1	 \\ 
8   & 6.8984e-01  & 		   8.5572e-04	&  2  \\   
7   & 6.8984e-01  & 	   8.5572e-04		&  3 \\  
10  & 4.8713e-01  & 		   1.0786e-03	&  4  \\  
11  & 4.8690e-01  & 		   1.1123e-03	&  	5   \\  
6   & 4.8613e-01  & 		   1.2438e-03	&  6   \\ 
12  & 4.8586e-01  & 	   1.2849e-03		& 	7   \\ 
13  & 4.8586e-01  & 	   1.3296e-03		&  	8 \\ 
14  & 4.7559e-01  &      1.6548e-03		&	9 \\ 
4   & 2.1282e-02	  & 	   4.0854e+00		&	  10\\ 
5   & 0 			  & 		   6.8384e+01	&   11 \\     
1   & 0			  & 		 1.1584e+02  	& 	     12\\
2   & 0 			  & 		   1.2971e+02	&		  13 \\  
3   & 0 			  & 	   2.6984e+03		& 	   14\\  
  \hline
\hline
\end{tabular}
\caption{\footnotesize Ranking of nodes based on the likelihood of having a price spike, according to both Monte Carlo simulation (in terms of probabilities $\widehat{\mathbb{P}}(Y_i)$) and large deviations results (in terms of decay rates $ I^*_i $), for the case $\bmuth^{(\text{high})} = 0.5 \times\bmu^{(\text{installed})}$, $\errrel=0.25$, $q=0.018$. The values $\widehat{\mathbb{P}}(Y_i)$, for $i = 1,2,3,5$, are not reported as the Monte Carlo simulation is not sufficiently accurate for such small probabilities. }
\label{tab:_ranking_1}
\end{table}   
%
\section{Concluding remarks and future work}\label{s:conclusion}
In this paper, we illustrate the potential of concepts from large deviations theory to study the events of rare price spikes caused by fluctuations of renewable generation. 
 By assuming a centralized perspective, we are able to use large deviations theory to approximate the probabilities of such events, and to rank the nodes of the power grids according to their likelihood of experiencing a price spike. Our technical approach is able to handle the multimodality of LMP's distributions, as well as violations of the LICQ regularity condition.
 
Future research directions include extending the present framework to non-Gaussian fluctuations, as well as incorporating a source of discrete noise in the form of line outages. Moreover, it would be of interest to study the sensitivity of the approximation in Eq.~\eqref{eq:ldp_approx} with respect to the tuning parameter $\lambda$, which quantifies the conservatism in the choice of the line limits. This would allow us to extend the notion of safe capacity regions~\cite{Nesti2019temperature,Nesti2017line} in the context of energy prices.
An alternative approach to deal with non-Gaussian fluctuations and more involved price spike structures could be to efficiently sample \textit{conditionally} on a price spike to have occurred, a problem for which specific Markov chain Monte Carlo (MCMC) methods have been developed, e.g., the Skipping Sampler~\cite{Moriarty2019}. In the case of a complicated multi-modal conditional distribution, the large deviations results derived in this paper can be of extreme help in identifying all the relevant price spikes modes, thus speeding up the MCMC procedure.
  
\vskip6pt




%

\paragraph{Acknowledgements}
This research is supported by NWO VICI grant 639.033.413, NWO Rubicon grant 680.50.1529 and EPSRC grant EP/P002625/1. The authors would like to thank the Isaac Newton Institute for Mathematical Sciences for support and hospitality during the programme ``The mathematics of energy systems'' when work on this paper was undertaken. This work was supported by EPSRC grant number EP/R014604/1.




%
%
%
%
%

%
\bibliographystyle{abbrv}
\bibliography{BiblioThesisCameraReadyCompletePaper}

\begin{thebibliography}{10}

\bibitem{Bienstockbook}
D.~Bienstock.
\newblock {\em Electrical transmission system cascades and vulnerability - an
  operations research viewpoint}, volume~22 of {\em {MOS-SIAM} Series on
  Optimization}.
\newblock {SIAM}, 2016.

\bibitem{Bo2009}
R.~Bo and F.~Li.
\newblock Probabilistic lmp forecasting considering load uncertainty.
\newblock {\em IEEE Transactions on Power Systems}, 24(3):1279--1289, 2009.

\bibitem{Bo2012}
R.~Bo and F.~Li.
\newblock Probabilistic lmp forecasting under ac optimal power flow framework:
  Theory and applications.
\newblock {\em Electric Power Systems Research}, 88(Supplement C):16 -- 24,
  2012.

\bibitem{Bucklew1990}
J.~A. Bucklew.
\newblock {\em Large deviation techniques in decision, simulation, and
  estimation}, volume 190.
\newblock Wiley New York, 1990.

\bibitem{Bushnell1999}
J.~Bushnell.
\newblock Transmission rights and market power.
\newblock {\em The Electricity Journal}, 12(8):77 -- 85, 1999.

\bibitem{Dembo1998}
A.~Dembo and O.~Zeitouni.
\newblock {\em Large deviations techniques and applications}.
\newblock Springer, 1998.

\bibitem{Dobkin1990contour}
D.~P. Dobkin, A.~R. Wilks, S.~V. Levy, and W.~P. Thurston.
\newblock Contour tracing by piecewise linear approximations.
\newblock {\em ACM Transactions on Graphics (TOG)}, 9(4):389--423, 1990.

\bibitem{Ferc2003}
FERC.
\newblock White paper wholesale power market platform, 2003.

\bibitem{Geng2016}
X.~Geng and L.~Xie.
\newblock Learning the lmp-load coupling from data: A support vector machine
  based approach.
\newblock {\em IEEE Transactions on Power Systems}, 32(2):1127--1138, 2017.

\bibitem{Gerster2016}
A.~Gerster.
\newblock Negative price spikes at power markets: the role of energy policy.
\newblock {\em Journal of Regulatory Economics}, 50(3):271--289, 2016.

\bibitem{Gonzales2017}
J.~Gonzalez, J.~Moriarty, and J.~Palczewski.
\newblock Bayesian calibration and number of jump components in electricity
  spot price models.
\newblock {\em Energy Economics}, 65:375 -- 388, 2017.

\bibitem{Hagfors2016}
L.~I. Hagfors, H.~H. Kamperud, F.~Paraschiv, M.~Prokopczuk, A.~Sator, and
  S.~Westgaard.
\newblock Prediction of extreme price occurrences in the german day-ahead
  electricity market.
\newblock {\em Quantitative Finance}, 16(12):1929--1948, 2016.

\bibitem{MPT3}
M.~Herceg, M.~Kvasnica, C.~Jones, and M.~Morari.
\newblock {Multi-Parametric Toolbox 3.0}.
\newblock In {\em Proc.~of the European Control Conference}, pages 502--510,
  2013.
\newblock \url{http://control.ee.ethz.ch/~mpt}.

\bibitem{Hoffmann2019consistency}
F.~Hoffmann, B.~Hosseini, Z.~Ren, and A.~M. Stuart.
\newblock Consistency of semi-supervised learning algorithms on graphs: Probit
  and one-hot methods.
\newblock {\em arXiv preprint arXiv:1906.07658}, 2019.

\bibitem{Hunueault1991}
M.~{Huneault} and F.~D. {Galiana}.
\newblock A survey of the optimal power flow literature.
\newblock {\em IEEE Transactions on Power Systems}, 6(2):762--770, 1991.

\bibitem{Ji2017}
Y.~Ji, R.~J. Thomas, and L.~Tong.
\newblock Probabilistic forecasting of real-time lmp and network congestion.
\newblock {\em IEEE Transactions on Power Systems}, 32(2):831--841, 2017.

\bibitem{Li2009}
F.~Li and R.~Bo.
\newblock Congestion and price prediction under load variation.
\newblock {\em IEEE Transactions on Power Systems}, 24(2):911--922, 2009.

\bibitem{Lu2005}
X.~Lu, Z.~Y. Dong, and X.~Li.
\newblock Electricity market price spike forecast with data mining techniques.
\newblock {\em Electric Power Systems Research}, 73(1):19 -- 29, 2005.

\bibitem{Moriarty2019}
J.~Moriarty, J.~Vogrinc, and A.~Zocca.
\newblock The skipping sampler: A new approach to sample from complex
  conditional densities.
\newblock {\em arXiv preprint arXiv:1905.09964}, 2019.

\bibitem{Nesti2019temperature}
T.~{Nesti}, J.~{Nair}, and B.~{Zwart}.
\newblock Temperature overloads in power grids under uncertainty: A large
  deviations approach.
\newblock {\em IEEE Transactions on Control of Network Systems},
  6(3):1161--1173, 2019.

\bibitem{Nesti2019blackout}
T.~Nesti, F.~Sloothaak, and B.~Zwart.
\newblock Emergence of scale-free blackout sizes in power grids.
\newblock In preparation.

\bibitem{Nesti2017line}
T.~Nesti, A.~Zocca, and B.~Zwart.
\newblock Line failure probability bounds for power grids.
\newblock In {\em 2017 IEEE Power \& Energy Society General Meeting}, pages
  1--5. IEEE, 2017.

\bibitem{Nesti2018emergent}
T.~Nesti, A.~Zocca, and B.~Zwart.
\newblock Emergent failures and cascades in power grids: A statistical physics
  perspective.
\newblock {\em Phys. Rev. Lett.}, 120:258301, 2018.

\bibitem{Neuhoff2011}
K.~Neuhoff, B.~F. Hobbs, and D.~Newbery.
\newblock {Congestion Management in European Power Networks: Criteria to Assess
  the Available Options}.
\newblock Discussion Papers of DIW Berlin 1161, German Institute for Economic
  Research, 2011.

\bibitem{Paraschiv2014}
F.~Paraschiv, D.~Erni, and R.~Pietsch.
\newblock The impact of renewable energies on eex day-ahead electricity prices.
\newblock {\em Energy Policy}, 73:196 -- 210, 2014.

\bibitem{Paraschiv2016}
F.~Paraschiv, R.~Hadzi-Mishev, and D.~Keles.
\newblock Extreme value theory for heavy tails in electricity prices.
\newblock {\em Journal of Energy Markets}, 9(2), 2016.

\bibitem{Purchala2005}
K.~{Purchala}, L.~{Meeus}, D.~{Van Dommelen}, and R.~{Belmans}.
\newblock Usefulness of {DC} power flow for active power flow analysis.
\newblock In {\em IEEE Power Engineering Society General Meeting, 2005}, pages
  454--459, 2005.

\bibitem{Radovanovic2019holistic}
A.~{Radovanovic}, T.~{Nesti}, and B.~{Chen}.
\newblock A holistic approach to forecasting wholesale energy market prices.
\newblock {\em IEEE Transactions on Power Systems}, 34(6):4317--4328, 2019.

\bibitem{Ren212019}
REN21.
\newblock Renewables 2016 global status report.
\newblock Technical report, Ren21 Secretariat, 2019.

\bibitem{SPPreport2016}
{Southwest Power Pool}.
\newblock 2016 annual state of the market report.
\newblock \url{https://www.spp.org/documents/53549/spp_mmu_asom_2016.pdf}.

\bibitem{Sun2010}
J.~Sun and L.~Tesfatsion.
\newblock {DC} optimal power flow formulation and solution using quadprogj.
\newblock 2010.

\bibitem{Tang2013nash}
W.~Tang and R.~Jain.
\newblock A nash equilibrium need not exist in the locational marginal pricing
  mechanism.
\newblock {\em ArXiv e-prints}, 2013.

\bibitem{Tondel2003}
P.~T{\o}Ndel, T.~A. Johansen, and A.~Bemporad.
\newblock An algorithm for multi-parametric quadratic programming and explicit
  mpc solutions.
\newblock {\em Automatica}, 39(3):489--497, 2003.

\bibitem{Touchette2009}
H.~Touchette.
\newblock The large deviation approach to statistical mechanics.
\newblock {\em Physics Reports}, 478(1):1 -- 69, 2009.

\bibitem{Veraart2016}
A.~E.~D. Veraart.
\newblock Modelling the impact of wind power production on electricity prices
  by regime-switching l{\'e}vy semistationary processes.
\newblock In {\em Stochastics of Environmental and Financial Economics}.
  Springer International Publishing, 2016.

\bibitem{Weron2014}
R.~Weron.
\newblock Electricity price forecasting: A review of the state-of-the-art with
  a look into the future.
\newblock {\em International Journal of Forecasting}, 30(4):1030 -- 1081, 2014.

\bibitem{Wu2006}
F.~Wu, F.~Zheng, and F.~Wen.
\newblock Transmission investment and expansion planning in a restructured
  electricity market.
\newblock {\em Energy}, 31(6):954 -- 966, 2006.
\newblock Electricity Market Reform and Deregulation.

\bibitem{Zhou2011}
Q.~Zhou, L.~Tesfatsion, and C.~C. Liu.
\newblock Short-term congestion forecasting in wholesale power markets.
\newblock {\em IEEE Transactions on Power Systems}, 26(4):2185--2196, 2011.

\bibitem{Zimmerman2011}
R.~D. Zimmerman, C.~E. Murillo-S{\'a}nchez, and R.~J. Thomas.
\newblock Matpower: Steady-state operations, planning, and analysis tools for
  power systems research and education.
\newblock {\em Power Systems, IEEE Transactions on}, 26(1):12--19, 2011.

\end{thebibliography}
\end{document}